\documentclass[11pt,a4paper]{article}
\usepackage{amsmath,amsfonts,amstext,amssymb,amscd,bezier,amsthm,latexsym,amsxtra}
\newtheorem{conj}{Conjecture}
\newtheorem{teo}{Theorem}
 
 \newtheorem{lem}{Lemma}
 \newtheorem{prop}{Proposition}
 \theoremstyle{definition}
 
 \theoremstyle{remark}

 \numberwithin{equation}{section}

\newcommand{\R}{\mathbb{R}} 
\newcommand{\C}{\mathbb{C}} 
\newcommand{\Q}{\mathbb{Q}} 
\newcommand{\h}{\mathbb{H}} 
\newcommand{\F}{\mathbb{F}} 

\begin{document}
\baselineskip=1.7em 

\title{A note on nontrivial intersection for selfmaps of complex Grassmann manifolds} \date{}
\author{Tha\'is F. M. Monis \\ tfmonis@rc.unesp.br \and Northon C. L. Penteado \\ northoncanevari@gmail.com  \and S\'ergio T. Ura \\ sergioura@gmail.com  \and
Peter Wong\thanks{This work was supported by Capes of Brazil - Programa Pesquisador Visitante Especial - Ci\^encia sem fronteiras, Grant number 88881.068085/2014-01.} \\
pwong@bates.edu}
\maketitle

\begin{abstract} Let $G(k,n)$ be the complex Grassmann manifold of $k$-planes in $\mathbb C^{k+n}$. In this note, we show that for $1<k<n$ and for any selfmap $f:G(k,n)\to G(k,n)$, there exists a $k$-plane $V^k\in G(k,n)$ such that $f(V^k)\cap V^k\ne \{0\}$.
\end{abstract}

\section{Introduction}

The problem of determining the fixed point property (f.p.p.) for Grassmann manifolds has been studied by many authors (for example \cite{oneill}, \cite{GH3}, \cite{Ho1}). 

Let
\[
\F M(n_1, \ldots , n_k) = \frac{U_\F (n)}{U_\F (n_1) \times \cdots \times U_\F (n_k)}, 
\]
$n_1+ \cdots + n_k =n$. Here, $\F$ stands for one of  the fields $\R$, $\C$ or the skew field $\h$, and 
 
\[
  U_\F (n) = \left\{\begin{array}{l}
  O(n) \ \text{the orthogonal group of order} \ n \ \text{if} \ \F=\R, \\
  U(n)  \ \text{the unitary group of order} \ n \ \text{if} \ \F=\C, \\
  Sp(n) \ \text{the symplectic group of order} \ n \ \text{if} \ \F=\h. \end{array} \right.
\]

\

In \cite{GH2}, Glover and Homer have given the following necessary condition for $\F M(n_1, \ldots , n_k)$ to have the f.p.p..

\begin{teo}[\cite{GH2}, Theorem 1] If $\F M(n_1, \ldots , n_k)$ has the f.p.p., then $n_1, \ldots , n_k$ are distinct integers and, if $\F=\R$ or $\C$, at most one is odd.
\end{teo}
\

The above theorem gives rise to the following conjectures:

\

\begin{conj} 
If $n_1, \ldots , n_k$ are all distinct then $\h M(n_1, \ldots , n_k)$ has the f.p.p..
\end{conj}

\begin{conj} 
If $n_1, \ldots , n_k$ are all distinct and at most one is odd then $\F M(n_1, \ldots , n_k)$ has the f.p.p., for $\F=\R$ and $\F=\C$.
\end{conj}

\

The above conjectures were already proved to be true in the following cases:

\begin{itemize}
\item Projective spaces ($\F M(1, n-1)$);

\item If $n_2$ and $n_3$ are distinct positive even integers and $n_3 \geq 2 n_2^2-1$ then $\C M (1, n_2, n_3)$ has the f.p.p. (\cite{GH2}).

\item If $1, n_2$ and $n_3$ are distinct positive integers and $n_3 \geq 2 n_2^2 -1$, then $\h M(1, n_2 , n_3)$ has the f.p.p. (\cite{GH2}).

\item If $n_2 < n_3$ are even integers greater than $1$ and either $n_2 \leq 6$ or $n_3 \geq n_2^2 - 2 n_2 -2$, then $\R M (1, n_2, n_3)$ has the f.p.p. (\cite{GH2}).

\item If $n_1, n_2 , n_3$ are positive integers such that at most one is odd, $n_1 \leq 3$, $n_3 \geq n_2^2 -1$, and $[n_1/2] < [n_2/2]< [n_3/2]$, then $\R M(n_1, n_2, n_3)$ has the f.p.p. (\cite{GH2}).

\item If $\F = \C$ or $\h$, $\F M (2, q) $ has the f.p.p. for all $q >2$ (\cite{oneill}).

\item $\R (2, q)$ has the f.p.p. for all $q=4k$ or $q=4k +1$, $k=1, 2, 3, \ldots$  (\cite{oneill}).

\item  For $p \leq 3$ and $q>p$ or $p>3$ and $q \geq 2p^2 -p - 1$, $\C M(p, q)$ has the f.p.p. iff $pq$ is even (\cite{GH3}).

\item For $p \leq 3$ and $q>p$ or $p>3$ and $q \geq 2p^2 -p - 1$, $\h M (p,q)$ always has the f.p.p. (\cite{GH3}). 
\end{itemize}

The main tool used to prove the above results is the calculation of the Lefschetz number of a self-map of such a space. Let's focus on the case of complex Grassmann manifolds $\C (k,n)=G(k,n)$, the space of $k$-planes in $\C^{k+n}$.   Let $\gamma^k$ be the canonical $k$-plane bundle over $G(k,n)$. If $$ch(\gamma^k)=1+c_1+ \cdots + c_k, \  \ c_i \in H^{2i}(G(k,n); \Q),$$ is the total Chern class of $\gamma^k$, then the cohomology ring $H^\ast(G(k,n); \Q)$  is given by: $$H^\ast(G(k,n); \Q) = \Q[c_1, \ldots , c_k]/ I_{k,n},$$ where $I_{k,n}$ is the ideal generated by the elements $(c^{-1})_{n+1}, \ldots , (c^{-1})_{n+k}$.  Here, $(c^{-1})_{q}$ is the part of the formal inverse of $c$ in dimension $2q$  (see \cite{Ho1}, Theorem 2.1). Then, $c_1$ is the only generator in dimension $2$. Therefore, given a self-map $f: G(k,n) \to G(k,n)$, $f^\ast (c_1)= m c_1$ for some coefficient $m$. 

\begin{teo}[\cite{GH3}, Theorem 1] 
Let $k \leq 3$ and $n>k$ or $k>3$ and $n \geq 2k^2-k-1$. Then every graded ring endomorphism of $H^\ast (G(k,n); \Q)$ is an Adams endomorphism\footnote{An Adams endomorphism of $H^\ast (G(k,n); \Q)$ is a endomorphism $\varphi$ of the form $\varphi(x)=\lambda^i x$ for $x\in H^{2i} (G(k,n); \Q)$. The coefficient $\lambda$ is called the degree of $\varphi$.}. Consequently, if $f: G(k,n) \to G(k,n)$ is a self-map with $f^\ast(c_1)=m c_1$ then $f^\ast (c_i)= m^i c_i$, $i=1, \ldots , k$. 
\end{teo}

The classification of the graded ring endomorphisms of $H^\ast (G(k,n); \Q)$ is fundamental in the study of f.p.p. for $G(k,n) $ because of the following.

\begin{prop} 
An Adams endomorphism of $H^\ast(G(k,n); \Q)$ has Lefschetz number zero if and only if its degree is $-1$ and $kn$ is odd.
\end{prop}
\begin{proof} See \cite{GH2}, Proposition 4.
\end{proof}

\

In \cite{Ho1}, M. Hoffman was able to prove the following.

\begin{teo}[\cite{Ho1}, Theorem 1.1] 
Let $k <n$ and $h$ be a graded ring endomorphism of $H^\ast (G(k,n); \Q)$ with $h(c_1)=m c_1$, $m \neq 0$. Then $h(c_i)=m^i c_i$, $1 \leq i \leq k$.
\end{teo}

\

If $k<n$ and $h$ is a graded ring endomorphism of $H^\ast (G(k,n); \Q)$ with $h(c_1)=0$, it is still unclear about what $h$ looks like in general. The conjecture is that, in this case, $h$ must be the null homomorphism. If one can prove such conjecture then the problem of determining the f.p.p. for $G(k,n)$ will be completely solved.

In this note, we prove a much more modest result for complex Grassmann manifolds than a fixed point theorem. Our main theorem is the following.

\

\begin{teo}[Main Result]  \label{TThm}
Let $k > 1$ and $k < n$. Then for every continuous map $f: G(k,n) \to G(k,n)$ there exists a $k$-plane $V^k \in G(k,n)$ such that $V^k \cap f(V^k) \neq \{0\}$.
\end{teo}

\

The motivation for this work is the paper \cite{Taghavi} where the author gave an alternative proof for the f.p.p. of $\C P^{2n}$ using characteristic classes. In fact, a closer look at the proof of the main result in  \cite{Taghavi} indicates that the same argument would also yield an alternative proof of the f.p.p. for $\mathbb RP^{2n}$ by replacing Chern classes with Stiefel-Whitney classes. We should also point out that a non-trivial intersection result similar to Theorem \ref{TThm} has been obtained in \cite{CS} for maps between two {\it different} Grassmann manifolds.

\section{Proof of the Main Theorem}

Throughout this paper, $G(k, n)$ denotes the complex Grassmann manifold  of  $k$-planes in $\C^{k+n}$.

Note that, since $G(k,n)$ and $G(n,k)$ are homeomorphic, $\gamma^k$ and $\gamma^n$ can be seen as subbundles of the trivial bundle $G(k,n) \times \C^{k+n}$, which is denoted by $\epsilon^{k+n}$, and, under such identification, 
\[
\gamma^k \oplus \gamma^n = \epsilon^{k+n}.
\]

\begin{lem} \label{lem1} 
Let $ch(\gamma^n)=1+\bar{c}_1+ \cdots + \bar{c}_n$
be the total Chern class of the bundle $\gamma^n$. Then, a general formula
for the class $\bar{c}_i$ in terms of the Chern classes of $\gamma^k$ is
given by  
\[
\bar{c}_i = \sum_{\|  \alpha \| = i} (-1)^{|\alpha|}
\frac{|\alpha|!}{\alpha !} ch(\gamma^k)^\alpha,
\]
where $\alpha$
represents the $k$-uple $\alpha=(a_1, \ldots , a_k)$, $\| \alpha \| =
a_1+2a_2+ \cdots +k a_k$, $|\alpha|=a_1+a_2+ \cdots + a_k$, $\alpha ! = a_1
! a_2 ! \cdots a_k !$ and $ch(\gamma^k)^\alpha = c_1^{a_1} \smile
c_2^{a_2} \smile \cdots \smile c_k^{a_k}$.
\end{lem}
\begin{proof} The proof is given recursively in the index $i$.

As $\gamma^k\oplus \gamma^n = \epsilon^{k+n}$, we have
\[
ch(\gamma^k)\smile ch(\gamma^n) = ch(\epsilon^{k+n}) = 1
\]
in $H^{\ast}(G(k,n); \mathbb Z)$. So \[ (1+ c_1+ \cdots + c_k)\smile
(1+\bar c_1+ \cdots + \bar c_n) = 1 \]
and then
\begin{eqnarray*}
1 & = & 1 \\
0 & = & c_1 +\bar c_1 \\
0 & = & c_2 +c_1\smile \bar c_1 + \bar c_2 \\
  &  & \cdots
\end{eqnarray*}

Then \[ \bar{c}_j = -\sum_{i=1}^{j} c_i\smile \bar{c}_{j-i} \]

for all $j=1,\dots , n$, with the convention $c_i=0$ when $i>k$. Thus,

\begin{description}

\item[(i)] $\bar{c}_1 = -c_1$;

\item[(ii)] $\bar{c}_2 = -(c_1\smile - c_1) -c_2 = c_1^2 - c_2$;

\item[(iii)] Suppose  \[ \bar{c}_j =
\sum_{||\alpha||=j}(-1)^{|\alpha|}\frac{|\alpha|!}{\alpha
!}ch(\gamma^k)^{\alpha} ,\]
for $j= 1,\dots , m-1 < n$.

\end{description}

\

Then
\begin{eqnarray*}
\bar c_m & =& -\sum_{i=1}^{m} c_i\smile \bar c_{m-i} \\
         & =& -\sum_{i=1}^{m} \left( c_i\smile
\sum_{||\alpha||=m-i}(-1)^{|\alpha|}
                                \frac{|\alpha|!}{\alpha !}ch(\gamma^k)^{\alpha} \right) \\
         & =& \sum_{i=1}^{m} \left( c_i\smile
\sum_{||\alpha||=m-i}(-1)^{|\alpha|+1}
                                \frac{|\alpha|!}{\alpha !}ch(\gamma^k)^{\alpha} \right) \\
         & = &\sum_{i=1}^{m} \left(  \sum_{||\alpha||=m-i}(-1)^{|\alpha|+1}
                                \frac{|\alpha|!}{\alpha !}ch(\gamma^k)^{\alpha}\smile c_i \right) \\
         & =& \sum_{i=1}^{m}  \sum_{||\alpha||=m-i}(-1)^{|\alpha+e_i|}
                                \frac{|\alpha|!}{\alpha !}ch(\gamma^k)^{\alpha + e_i} 
																\hspace{1cm} (e_i=(0,\dots,0,1,0,\dots 0)) \\
         & =& \sum_{||\beta||=m}(-1)^{|\beta|} X(\beta)
ch(\gamma^k)^{\beta}           \hspace{2.7cm}  (\beta = \alpha + e_i)
\end{eqnarray*}

where 
\begin{align*}
X(\beta) & = \sum_{b_i\neq 0} \frac{|\beta-e_i|!}{(\beta-e_i)!} \\
         & = \sum_{b_i\neq 0} \frac{(|\beta|-1)! b_i}{\beta!} \\
         & = \sum_{i=1}^{m} \frac{(|\beta|-1)!b_i}{\beta!} \\
         & = \frac{(|\beta|-1)!\sum_{i=1}^{m} b_i}{\beta!} \\
         & = \frac{(|\beta|-1)!|\beta|}{\beta!} \\
         & =  \frac{|\beta|!}{\beta!}
\end{align*}

\end{proof}

\

\subsection{Proof of Theorem \ref{TThm}}

Suppose, to the contrary, there exists a continuous map $f: G(k,n) \to G(k,n)$ such that $V^k \cap f(V^k) = \{ 0\}$ for every $k$-plane $V^k \in G(k,n)$. Then the direct sum $\gamma^k \oplus f^\ast \gamma^k$ can be seen as a subbundle of the trivial bundle $\epsilon^{k+n}$. Let $\eta^{n-k}$ be the normal bundle of $\gamma^k \oplus f^\ast \gamma^k$ in  $\epsilon^{k+n}$. Then
\begin{equation}
ch(\gamma^k) \smile ch(f^\ast \gamma^k) \smile ch(\eta^{n-k})=1.
\end{equation}

It follows that 
\begin{equation}
ch(f^\ast \gamma^k) \smile ch(\eta^{n-k})=1+\bar{c}_1+ \cdots + \bar{c}_n.
\end{equation}

Let
\begin{equation}
ch(f^\ast \gamma^k)=1+\tilde{c}_1+ \cdots + \tilde{c}_k, \  \tilde{c}_i \in H^{2i}(G(k,n); \Q),
\end{equation}
and
\begin{equation}
ch(\eta^{n-k})=1+t_1+ \cdots + t_{n-k}, \  t_{j} \in H^{2j}(G(k,n); \Q).
\end{equation}

We will show that it is impossible for
\begin{equation}
\bar{c}_n=\tilde{c}_k \smile t_{n-k}.
\end{equation}

The proof of the impossibility of the above equality will be split into several cases.

\

\noindent {\bf Case 1:  $1< k \leq 3$.}  Since $c_1 \in H^2(G(k,n); \Q)$ is the only generator in dimension $2$, $f^\ast(c_1)$ is a multiple of $c_1$, let's say $f^\ast(c_1) = m c_1$.  Following \cite{oneill} and \cite{GH3}, for $k \leq 3$ and $k<n$, every endomorphism of the ring $H^\ast(G(k,n); \Q)$ that preserves dimension is an Adams endomorphism. Therefore, if $f^\ast(c_1)=m c_1$ then $f^\ast (c_2)=m^2 c_2, \ldots , f^\ast(c_k)=m^k c_k$. Thus $$ch(f^\ast \gamma^k)=f^\ast (ch(\gamma^k)) = 1+m c_1 + m^2 c_2 + \cdots + m^k c_k.$$ It follows that $$\bar{c}_n = m^k c_k \smile t_{n-k},$$ in contradiction with Lemma \ref{lem1}. 

\

\noindent {\bf Case 2: $k>3$.} This case will be split in four cases.

\

{\bf Case 2(i):} $n=l(k-1)+r$ with remainder $r\neq 1$, that is, $1<r < k-1$ or $r=0$. In this case, $r$ is of the form $r=2i$ or $r=2i+3$, for some integer $i\geq 0$. In case of $r=2i$, the class $c_{k-1}^l \smile c_2^i$ does not appear in $\tilde{c}_k \smile t_{n-k}$ but, by Lemma \ref{lem1}, it appears in $\bar{c}_n$, contradicting $\bar{c}_n=\tilde{c}_k \smile t_{n-k}$.
In case of $r=2i+3$, the class $c_{k-1}^l c_2^i c_3$ does not appear in $\tilde{c}_k \smile t_{n-k}$ but, by Lemma \ref{lem1}, it appears in $\bar{c}_n$, contradicting $\bar{c}_n=\tilde{c}_k \smile t_{n-k}$.

\

{\bf Case 2(ii): $k>4$ and $n=(l+1)(k-1)+1$.} In this case, we have
\begin{eqnarray*}
n&=&(l+1)(k-1)+1 \\
&=& l(k-1)+k
\end{eqnarray*}
and, since $n >k$, $l \geq 1$. We can write $n=(l+1)(k-1)+1$ in the form 
\[ n=(l-1)(k-1)+2(k-2)+3 \] 
and, since we are supposing $k >4$, $k-2>2$. With these information, one can check that the class $c_{k-1}^{m-1} \smile c_{k-2}^2 \smile c_3$ cannot appear in $\tilde{c}_k \smile t_{n-k}$. On the other hand, by Lemma \ref{lem1}, the class $c_{k-1}^{m-1} \smile c_{k-2}^2 \smile c_3$ appears in $\bar{c}_n$. Therefore, $\bar{c}_n=\tilde{c}_k \smile t_{n-k}$ is impossible.

\

 {\bf Case 2(iii):} $k=4$, $n=(l+1)(k-1)+1$ and $l$ even, say $l=2j$. In this case, $n-k=3l$ and, since $n >1$, $l \geq 1$. Let 
 \begin{eqnarray*}
 \tilde{c}_4 &=& c_1^4 + \alpha c_2^2 + \theta c_4 + \text{other terms} \\
 t_{3l} &=& c_1^{3l} + \alpha' c_2^{3j} + \beta c_3^l + \text{other terms.}
 \end{eqnarray*}
 Thus, in the product $\tilde{c}_4 \smile t_{3l}$, $\alpha \alpha'$ is the coefficient of $c_2^{3j+2}$, $\alpha \beta$ is the coefficient of $c_2^2 \smile c_3^l$ and $\theta \beta$ is the coefficient of $c_4 \smile c_3^l$. From Lemma \ref{lem1} together with the fact that $\tilde{c}_4 \smile t_{3l} = \bar{c}_n$, it follows that
 \begin{eqnarray*}
 \alpha \alpha' &=& \frac{(3j+2)!}{(3j+2)!1!} \\
 \alpha \beta &=& \frac{(l+2)!}{l!2!} \\
 \theta \beta &=& \frac{(l+1)!}{l!1!}.
 \end{eqnarray*} 
Thus
\begin{eqnarray*}
\alpha \alpha' &=& 1 \\
\alpha \beta &=&\frac{(l+2)(l+1)}{2} \\
 \theta \beta &=& l+1.
\end{eqnarray*}

Then, we conclude that $\alpha = \pm 1$, $\displaystyle \beta = \pm \frac{(l+2)(l+1)}{2} $ and $\displaystyle |\beta| = \frac{(l+2)(l+1)}{2}$ divides $\theta \beta = l+1$. It follows that $l=0$, but $l \geq 1$, a contradiction!

\

 {\bf Case 2(iv):} $k=4$, $n=(l+1)(k-1)+1$ and $l$ odd, say $l=2j+1$. Again, $n-k=3l$ and, since $n >1$, $l \geq 1$. Let
 \begin{eqnarray*}
 \tilde{c}_4&=&c_1^4+ \alpha c_2^2+ \theta c_4 + \gamma c_1 c_3 + \text{other terms} \\
 t_{3l} &=& c_1^{3l} + \alpha' c_1 c_2^{3j+1} + \beta c_3^l + \text{other terms}.
 \end{eqnarray*}
It follows that, in the product $\tilde{c}_4 \smile t_{3l}$, $\alpha \alpha'$ is the coefficient of $c_1 \smile c_2^{3j+3}$, $\alpha \beta$ is the coefficient of $c_2^2 \smile c_3^{l}$, $\theta \beta$ is the coefficient of $c_4 \smile c_3^{l}$ and $\gamma \beta$ is the coefficient of $c_1 \smile c_3^{l+1}$. Since $\bar{c}_n=\tilde{c}_4 \smile t_{3l}$, together with Lemma \ref{lem1}, 
\begin{eqnarray*}
\alpha \alpha' &=& \frac{(3j+4)!}{1!(3j+3)!}\\
\alpha \beta &=& \frac{(l+2)!}{l!2!} \\
\theta \beta &=& \frac{(l+1)!}{l!1!} \\
\gamma \beta &=& \frac{(l+2)!}{1!(l+1)!}.
\end{eqnarray*}
 
Thus

\begin{eqnarray*}
\alpha \alpha' &=& 3j+4 \\
\alpha \beta &=&\frac{(l+2)(l+1)}{2} \\
 \theta \beta &=& l+1 \\
 \gamma \beta &=& l+2.
\end{eqnarray*}
 
From the two last equalities above, it follows that $\beta$ divides $l+1$ and $l+2$. Therefore, $\beta=1$. It follows that $\alpha = \frac{(l+2)(l+1)}{2}$ and, since $\alpha $ divides $3j+4$, 
\[ \frac{(l+2)(l+1)}{2} \leq 3j+4 = \frac{3l+5}{2}. \] 
Therefore, $l^2 \leq 3$. Since $l$ is an integer not smaller than $1$, it follows that $l=1$. Then, $3j+4=\frac{3l+5}{2}=4$ is divisible by $\frac{(l+2)(l+1)}{2}=3$, a contradiction!  
\qed

\newpage

{\bf Tha\'is F. M. Monis}

Departamento de Matem\'atica, IGCE, Univ Estadual Paulista.

email: tfmonis@rc.unesp.br

\

{\bf Northon C. L. Penteado}

Departamento de Matem\'atica, IGCE, Univ Estadual Paulista.

email: northoncanevari@gmail.com

\

{\bf S\'ergio T. Ura}

Departamento de Matem\'atica, IGCE, Univ Estadual Paulista.

email: sergioura@gmail.com

\

{\bf Peter Wong}

Department of Mathematics, Bates College

e-mail: pwong@bates.edu

\

\end{document}